\documentclass[reqno, 11pt]{amsart}
\usepackage{amsmath}
\usepackage{amssymb}
\usepackage{amsmath}
\usepackage{amssymb}
\usepackage{amssymb}
\usepackage{mathrsfs}
\usepackage{amsmath}
\usepackage{amsmath,amsthm,amssymb,amscd}
\usepackage{latexsym}
\usepackage{color}
\usepackage{graphicx}
\usepackage[colorlinks=true]{hyperref}
\hypersetup{urlcolor=blue, citecolor=blue,linkcolor=blue}
\newtheorem{theorem}{Theorem}[section]
\newtheorem{corollary}[theorem]{Corollary}

\newtheorem{remark}[theorem]{Remark}
\newtheorem{example}[theorem]{Example}

\theoremstyle{definition} \theoremstyle{remark}
\numberwithin{equation}{section}

\def\C{\mathbb C}
\def\R{\mathbb R}

\def\Z{\mathbb Z}

\def\Z{\mathbb Z}

\def\<{\langle}
\def\>{\rangle}

\def\diag {\mathrm{diag}}

\newcommand{\tr}{\operatorname{tr}}

\def\F{\mathbb{F}}
\renewcommand{\i}{\mathbf{i}}
\newcommand{\be}[2]{\begin{#1} #2 \end{#1}}  
\newcommand{\mtx}[1]{\begin{pmatrix} #1 \end{pmatrix}} 
\newcommand{\wt}[1]{\widetilde{#1}}
\renewcommand{\Re}{\operatorname{Re}}
\newcommand{\Ima}{\operatorname{Im}}
\newcommand{\ol}[1]{\overline{#1}}

\begin{document}

\title{Maps   preserving trace of products of matrices}

\author[H. Huang] {Huajun Huang}
\address{Department of Mathematics and Statistics, Auburn University,
AL 36849--5310, USA}
  \email{huanghu@auburn.edu}

\author[M. Tsai] {Ming-Cheng Tsai}
\address{General Education Center,
Taipei University of Technology, Taipei 10608, Taiwan}
\email{mctsai2@mail.ntut.edu.tw}

\begin{abstract}
We prove the linearity and injectivity of  two maps $\phi_1$ and $\phi_2$ on certain subsets of $M_n$ that satisfy $\tr(\phi_1(A)\phi_2(B))=\tr(AB)$.
We apply it to characterize  maps $\phi_i:\mathcal{S}\to \mathcal{S}$ ($i=1, \ldots, m$) satisfying
$$\tr (\phi_1(A_1)\cdots \phi_m(A_m))=\tr (A_1\cdots A_m)$$
 in which $\mathcal{S}$ is the set of $n$-by-$n$
general, Hermitian, or symmetric matrices for $m\ge 3$,  or   positive definite or diagonal matrices for $m\ge 2$.
The real versions  are also given.
\end{abstract}

%
%
%

\keywords{trace of products, preserver, linear operators, matrix space}

\footnote{Mathematics Subject Classification 2020: Primary 15A86, Secondary 47B49, 15A15, 15A04.}

\maketitle

\section{Introduction}

Preserver problem is one of the most active research area in matrix theory (e.g. \cite{Li01, Li92, Molnar, Pierce}). Researchers would like to characterize the maps on a given space of matrices preserving certain subsets, functions or relations. In this paper, we investigate maps $\phi_1,\ldots,\phi_m$ on several types of matrix spaces that have the trace-preserving property:
\be{equation}
{\label{20}
\tr (\phi_1(A_1)\cdots \phi_m(A_m))=\tr (A_1\cdots A_m).
}

Let $M_{m,n}$ denote the set of $m\times n$ complex matrices.
Let $M_n$, $H_n$, $S_n$, $P_n$, $\ol{P_n}$, and $D_n$ denote the set of $n\times n$ complex general, Hermitian, symmetric, positive definite, positive semi-definite, and diagonal matrices respectively. We add $(\R)$ after the above sets to denote their counterparts in real field.
Let $B(X)$ denote the space of all bounded linear operators on a Banach space $X$.

The classical Wigner's unitary-antiunitary theorem {\cite[p.12]{Molnar}} says that if a bijective map $\phi$ defined on the set of all rank one projections acting on a Hilbert space $H$ with the property
$$\tr(\phi(P)\phi(Q))=\tr(PQ),$$
that is, $\phi$ preserves the transition probability $\langle \phi(P),\phi(Q)\rangle=\langle P, Q\rangle$ between the pure states $P$ and $Q$ (or the angle between the ranges of $P$ and $Q$), then there is an either unitary or antiunitary operator $U$ on $H$ such that $\phi(P)=U^{*}PU$ or $\phi(P)=U^{*}P^{t}U$ for all rank one projections $P$. There is a nice generalization of Wigner's theorem due to Uhlhorn. He shows that the same conclusion holds if the property $\tr(\phi(P)\phi(Q))=\tr(PQ)$ is replaced by
$\tr(\phi(P)\phi(Q))=0 \Leftrightarrow \tr(PQ)=0$
(see \cite{Uh63}).

In 2002, linearity  of  maps  $\phi$ on $B(X)$ satisfying
\begin{equation}\label{tr-prod-2-preserving}
\tr(\phi(A)\phi(B))=\tr(AB),
\end{equation}
for $A\in B(X)$, rank one operator $B\in B(X)$ are explored by Moln\'{a}r in the proof of \cite[Theorem 1]{Molnar2} to characterize surjective maps preserving the point spectrum of product of operators:
$\sigma_p(\phi(A)\phi(B))=\sigma_p(AB)$ for $A, B \in B(X)$.
Later,  Moln\'{a}r  removes the surjectivity assumption and proves the bijectivity  of $\phi$  on the finite dimensional case $M_n$ in the proof of \cite[Theorem 1]{Molnar3}.

 In 2012, Li, Plevnik, and {\v S}emrl \cite{Li12} characterize bijective maps $\phi: S\rightarrow S$ such that for a given  real number $c$,
$$\tr(\phi(A)\phi(B))=c \ \Leftrightarrow \ \tr(AB)=c,$$
where $S$ is $H_n$, the set $S_n(\R)$ of   real symmetric matrices,  or the set of  rank one projections. In \cite[Lemma 3.6]{Huang16}, Huang et al show that the following statements are equivalent for
a unital operator $\phi$ on $P_n$:
\be{enumerate}{
\item
$\tr(\phi(A)\phi(B))=\tr(AB)$ for $A,B\in P_n$;
\item
$\tr(\phi(A)\phi(B)^{-1})=\tr(AB^{-1})$ for $A, B\in P_n$;
\item
$\det(\phi(A)+\phi(B))=\det(A+B)$ for $A,B\in P_n$;
\item
$\phi(A)=U^*AU$ or $U^*A^t U$ for a unitary matrix $U$.
}
 The authors  also consider the cases if $\phi$ is not assuming unital, the set $P_n$ is replaced by another set like $M_n$,  $S_n$ , or
 the set of  upper triangular matrices in $M_n$. In \cite[Theorem 3.8]{Leung16}, Leung, Ng, and Wong consider the relation \eqref{tr-prod-2-preserving} in terms of Raggio transition probability
between the spaces  of all normal states on  two von Neumann algebras.

Recently, some researchers consider similar relations of two maps preserving certain spectra on some spaces. See Miura in \cite{Mi09}  between two standard operator algebras. See Bourhim and Lee in \cite{Bourhim} between $B(X)$ and $B(Y)$ for some Banach spaces $X$ and $Y$. See Abdelali and Aharmim in \cite{Abdelali} between two rectangular matrices.

Let $I_n$ (or simply $I$) denote the identity matrix in $M_n$.
Given a subset $S$ of a vector space over a field $\F$, let $\langle S\rangle$   denote the $\F$-subspace spanned by elements of $S$.

In Section 2, we prove the linearity  and injectivity of
two maps $\phi_1$ and $\phi_2$ on some matrix sets $\mathcal{S}$   that satisfy \eqref{20}
for $m=2$ (Theorems \ref{thm: two maps preserving trace}
and \ref{thm: two maps preserving trace linear}):
\be{equation}{\label{trace-preserving-2-maps}
\tr(\phi_1(A)\phi_2(B))=\tr(AB),\quad \forall A, B\in \mathcal{S}.
}
  These preliminary results
are widely used in our characterizations of maps preserving trace of products.
The matrix sets discussed in this paper are  closed under Hermitian conjugate. The relation  \eqref{trace-preserving-2-maps} in these sets can also be interpreted as inner product preserver:
\be{equation}{
\langle \phi_1(A),\phi_2(B)\rangle=\langle A,B\rangle,\quad \forall A, B\in \mathcal{S}.
}
Our discussions and results on trace preserving maps will be helpful to characterize inner product preservers in related matrix spaces.

The maps $\phi_1,\ldots,\phi_m:\mathcal{S}\to \mathcal{S}$ that satisfy \eqref{20} are characterized as follows:

$m\geq3$:
\begin{itemize}
\item $\mathcal{S}=M_n$: $\phi_i(A)=N_i A N_{i+1}^{-1}$ where $N_{m+1}:=N_1$ (Theorem \ref{thm: M_n trace preserve m product}).

\item  $\mathcal{S}= H_n$ (resp., $P_n$ or $\ol{P_n}$):
\begin{enumerate}
\item
when $m$ is odd, $\phi_i(A)=c_i U^*AU$  where $U$ is unitary;
\item
when $m$ is even, $\phi_i(A)=c_i M^*AM$ if $i$ is odd, and $\phi_i(A)=c_i M^{-1}AM^{-*}$ if $i$ is even.
\end{enumerate}
Here $c_i\in \mathbb{R}$ (resp., $c_i\in \mathbb{R}^+$) with $c_1\cdots c_m=1$, as stated in Theorem \ref{thm: H_n trace preserve m product} (resp. Theorem \ref{thm: multiple product trace preserver}).

\item  $\mathcal{S}=S_n$:
\begin{enumerate}
\item
when $m$ is odd, $\phi_i(A)=c_i O^tAO$  where $O$ is orthogonal;
\item
when $m$ is even, $\phi_i(A)=c_i M^tAM$ if $i$ is odd, and $\phi_i(A)=c_i M^{-1}AM^{-t}$ if $i$ is even.
\end{enumerate}
Here $c_i\in \mathbb{C}$ with $c_1\cdots c_m=1$, as stated in Theorem \ref{thm: S_n trace preserve m product}.

\item    $\mathcal{S}=D_n$: $\phi_i (A)=C_i P^t AP$  where $P$ is a permutation matrix, and $C_1,\ldots,C_m$ are diagonal matrices with $C_1\cdots C_m=I$ (Theorem \ref{thm: D_n trace preserve m product}).

\end{itemize}

$m=2$:

\begin{itemize}
\item    $\mathcal{S}=P_n$ or $\overline{P_n}$: $\phi_1(A)=M^*AM,\  \phi_2(A)=M^{-1} AM^{-*}$; or $\phi_1(A)=M^*A^tM,\ \phi_2(A)=M^{-1} A^tM^{-*}$ (Theorem \ref{thm: P_n preserve trace}).

\item    $\mathcal{S}=D_n$: $\phi_1(A)=\diag(N\diag^{-1}(A)),\  \phi_2(A)=\diag(N^{-t}\diag^{-1}(A))$ (Theorem \ref{thm: trace preserve 2 product}).

\end{itemize}

Corollary \ref{thm: multiple weighted product trace preserver} describes maps $\phi_i: P_n\to \ol{P_n}$ ($i=1,\ldots,m$, $m\ge 2$) that satisfy
\begin{equation}
\tr \left(\phi_1(A_1)^{\alpha_1}\cdots \phi_m(A_m)^{\alpha_m}\right)=\tr \left(A_1^{\beta_1}\cdots A_m^{\beta_m}\right),\quad\forall A_1,\ldots,A_m\in P_n,
\end{equation}
for given nonzero real numbers $\alpha_1,\ldots,\alpha_m, \beta_1,\ldots,\beta_m.$

All above results have the corresponding real versions. We discuss them after the complex versions in each section.

\section{Preliminary}

\subsection{Two maps preserving trace of product}


We first give a general result about two maps preserving   trace of product.  The theorem is a powerful  tool to prove linear bijectivity
of various maps preserving   trace of product in
our later discussions. Recall that if $S$ is a subset of a vector space, then
$\langle S\rangle$ denotes the subspace spanned by $S$.

\begin{theorem}\label{thm: two maps preserving trace}
Let  $\phi: V_1\to W_1$ and $\psi: V_2\to W_2$  be two maps between subsets of matrix spaces over a field $\F$ such that:
\begin{enumerate}
\item $\dim \langle V_1\rangle=\dim \langle V_2\rangle\ge \max\{\dim \langle W_1\rangle, \dim \langle W_2\rangle\}$.
\item 
$AB$  are well-defined square matrices for $(A,B)\in (V_1\times V_2)\cup(W_1\times W_2)$.
\item \label{tr invertible}
 If $A\in \langle V_1\rangle $ satisfies that $\tr(AB)=0$ for all $B\in \langle V_2\rangle$, then $A=0$.
\item $\phi$ and $\psi$ satisfy that
\begin{equation}
\tr (\phi(A)\psi(B))=\tr(AB),\quad A\in V_1,\ B\in V_2.
\end{equation}
\end{enumerate}
Then
$\dim \langle V_1\rangle =\dim \langle V_2\rangle =\dim \langle W_1\rangle =\dim \langle W_2\rangle $
and $\phi$ and $\psi$ can be extended to bijective linear map $\widetilde\phi : \langle V_1\rangle \to  \langle W_1\rangle$ and $\widetilde\psi: \langle V_2\rangle \to  \langle W_2\rangle$, respectively, such that
\begin{equation}\label{extended map preserve trace}
\tr (\wt{\phi}(A)\wt{\psi}(B))=\tr(AB),\quad   A\in \langle V_1\rangle,\ B\in \langle V_2\rangle.
\end{equation}
\end{theorem}

\begin{proof}
Let $\{A_1,\ldots,A_m\}\subseteq V_1$ be a basis of $\langle V_1\rangle $, and $\{B_1,\ldots, B_m\}\subseteq V_2$ a basis of $\langle V_2\rangle$, respectively. By assumptions (3) and (4), the matrix
\begin{equation}\label{eq: two maps trace of bases}
[\tr(\phi(A_i)\psi(B_j))]_{m\times m}=[\tr(A_i B_j)]_{m\times m}
\end{equation}
 is invertible. So
both sets $\{\phi(A_1),\ldots,\phi(A_m)\}\subseteq W_1$ and $\{\psi(B_1),\ldots,\psi(B_m)\}\subseteq W_2$
are linearly independent.
By assumption (1), $\dim \langle V_1\rangle =\dim \langle V_2\rangle =\dim \langle W_1\rangle=\dim \langle W_2\rangle$,
and
$\{\phi(A_1),\ldots,\phi(A_m)\}$ (resp.  $\{\psi(B_1),\ldots,\psi(B_m)\}$) is a basis of $W_1$ (resp. $W_2$.)

Let
$\widetilde\phi : \langle V_1\rangle \to  \langle W_1\rangle$ and $\widetilde\psi: \langle V_2\rangle \to  \langle W_2\rangle$ be the   linear maps defined  in the way that for all $c_1,\ldots,c_m\in\F$:
\begin{equation}
\wt{\phi}\left(\sum_{i=1}^{m} c_i A_i\right) :=\sum_{i=1}^{m} c_i \phi(A_i),
\qquad
\wt{\psi}\left(\sum_{i=1}^{m} c_i B_i\right) :=\sum_{i=1}^{m} c_i \psi(B_i).
\end{equation}
If a linear combination $\sum_{i=1}^{m} c_i A_i\in V_1$, then for all $j=1,\ldots,m$,
\begin{eqnarray*}
\tr\left[\wt{\phi}\left(\sum_{i=1}^{m} c_i A_i\right) \psi(B_j) \right]
&=&
\sum_{i=1}^{m} c_i \tr\left( \phi(A_i) \psi(B_j) \right)
\\
&=& \sum_{i=1}^{m}c_i \tr( A_iB_j)=
\tr \left[\left(\sum_{i=1}^{m}c_i  A_i\right) B_j\right]
\\
&=& \tr\left[\phi\left(\sum_{i=1}^{m} c_i A_i\right) \psi(B_j) \right].
\end{eqnarray*}
So by assumption (3) and \eqref{eq: two maps trace of bases}, $\wt{\phi}\left(\sum_{i=1}^{m} c_i A_i\right)=\phi\left(\sum_{i=1}^{m} c_i A_i\right)$.
Therefore,   $\wt{\phi}: \langle V_1\rangle \to  \langle W_1\rangle$ is a  linear bijection and an extension of $\phi:V_1\to W_1$. Similarly,
 $\wt{\psi}: \langle V_2\rangle \to  \langle W_2\rangle$ is a linear bijection and an extension of $\psi:V_2\to W_2$.
\end{proof}

Theorem \ref{thm: two maps preserving trace} can be applied to \cite[Lemma 3.5]{Abdelali} which says that:  if $U$ and $V$ are nonempty open subsets of $M_{m,n}$ and
$M_{n,m}$ respectively, $\phi:U\to M_{p,q}$ and $\psi:V\to M_{q,p}$ satisfy that $\tr(\phi(A)\psi(B))=\tr(AB)$ for all $A\in U$ and $B\in V$, and $mn\ge pq$, then
$\phi$ and $\psi$ can be extended to linear bijections $M_{m,n} \to M_{p,q}$ and $M_{n,m}\to M_{q,p}$ respectively.
Theorem \ref{thm: two maps preserving trace} on maps $M_n\to M_n$ also generalizes the result: if a map $\phi:M_n\to M_n$ satisfies $\tr(\phi(A)\phi(B))=\tr (AB)$ for all $A, B\in M_n$, then $\phi$ must be a linear bijection in Moln\'{a}r \cite[proof of Theorem 1]{Molnar3} and in J. T. Chan, C. K. Li, and N. S. Sze  \cite[Prop. 1.1]{ChanLiSze}.

A subset $V$ of $M_n$   is closed under Hermitian conjugate if $\{A^*: A\in V\}\subseteq V$.
The following are some examples:
\begin{enumerate}
\item Each of $M_n$, $H_n$, $P_n$, $\ol{P_n}$,  $D_n$,  and the sets  of real/complex symmetric/skew-symmetric matrices,
is closed under Hermitian conjugate.
\item A combination of intersections, unions, sums, direct sums,  Kronecker products and compounds  of matrices in some sets closed under Hermitian conjugate will generate a set closed under Hermitian conjugate.
\end{enumerate}
A real or complex  matrix space $V$ is closed under Hermitian conjugate if and only if $V$ equals the direct sum of
 subspace of Hermitian  matrices and  subspace of skew-Hermitian  matrices.

\begin{corollary}\label{thm: trace preserver on special spaces}
Let $V$ be a subset of $M_n$ closed under Hermitian conjugate.
Suppose two maps $\phi,\psi:V\to V$   satisfy that
\begin{equation}\label{general two maps preserve trace}
\tr(\phi(A)\psi(B))=\tr(AB),\quad  A, B\in V.
\end{equation}
Then $\phi$ and $\psi$   can be extended to linear bijections on $\langle V\rangle$.
Moreover, when $V$ is a vector space,
 every linear bijection $\phi:V\to V$ corresponds to a unique  linear bijection $\psi:V\to V$ such that \eqref{general two maps preserve trace} holds.
Explicitly, given   an orthonormal basis $\{A_1,\ldots,A_\ell\}$ of $V$ with respect to the inner product $\langle A,B\rangle =\tr(A^*B)$,
$\psi$ is defined by
 $\psi(A_i) =B_i$  in which $\{ B_1,\ldots, B_\ell\}$
 is a basis of $V$ with $\tr(\phi(A_i^*) B_j)=\delta_{i,j}$ for all $i,j\in\{1,\ldots,\ell\}$.
\end{corollary}


A similar result of Theorem \ref{thm: two maps preserving trace} is obtained by assuming linearity of $\phi$ and $\psi$ but
dropping the dimension restriction on their codomains, in which $\phi$ and $\psi$ can be extended to injective linear maps preserving trace of product. A map $\phi: V\to W$ between subsets of vector spaces over $\F$ is called linear if
$v_1,\ldots,v_k, c_1v_1+\cdots+c_kv_k\in V$ in which $c_1,\ldots,c_k\in\F$ implies that
\be{equation}{
\phi( c_1v_1+\cdots+c_kv_k)= c_1\phi(v_1)+\cdots+c_k\phi(v_k).
}

\begin{theorem}\label{thm: two maps preserving trace linear}
Let  $\phi: V_1\to W_1$ and $\psi: V_2\to W_2$  be two linear maps between subsets of matrix spaces over a field $\F$ such that:
\begin{enumerate}
\item $\dim \langle V_1\rangle=\dim \langle V_2\rangle$.
\item 
$AB$  are well-defined square matrices for $(A,B)\in (V_1\times V_2)\cup(W_1\times W_2)$.
\item
 If $A\in \langle V_1\rangle $ satisfies that $\tr(AB)=0$ for all $B\in \langle V_2\rangle$, then $A=0$.
\item $\phi$ and $\psi$ satisfy that
\begin{equation}
\tr (\phi(A)\psi(B))=\tr(AB),\quad   A\in V_1,\ B\in V_2.
\end{equation}
\end{enumerate}
Then
$\dim \langle V_1\rangle =\dim \langle V_2\rangle \le\min\{\dim \langle W_1\rangle,\dim \langle W_2\rangle\}  $
and $\phi$ and $\psi$ can be extended to  two injective linear maps $\widetilde\phi : \langle V_1\rangle \to  \langle W_1\rangle$ and $\widetilde\psi: \langle V_2\rangle \to  \langle W_2\rangle$, respectively, such that
\begin{equation}\label{extended map preserve trace}
\tr (\wt{\phi}(A)\wt{\psi}(B))=\tr(AB),\quad  A\in \langle V_1\rangle,\ B\in \langle V_2\rangle.
\end{equation}
\end{theorem}

\be{proof}{
We have $\dim \langle V_1\rangle=\dim \langle V_2\rangle\ge \max\{\dim\langle \Ima\phi\rangle,\dim\langle \Ima\psi\rangle\}$.
View $\phi$ and $\psi$ as $\phi:V_1\to  \Ima\phi$ and $\psi: V_2\to \Ima\psi$ and apply Theorem \ref{thm: two maps preserving trace}.
}

\begin{corollary} \label{thm: trace preserving injection}
 Let $V$ and $W$ be subsets of matrix spaces in which $V$ is closed under Hermitian conjugate.
 Suppose two linear maps $\phi, \psi: V\to W$  satisfy that
\begin{equation}\label{matrix algebra trace preserver}
\tr(\phi(A)\psi(B))=\tr(AB),\quad A, B\in V.
\end{equation}
Then $\phi$ and $\psi$ can be extended to  injective linear maps $\widetilde\phi, \widetilde\psi: \langle V\rangle \to  \langle W\rangle$  preserving trace of product. In particular, $\dim \langle V\rangle \le\dim \langle W\rangle$.
\end{corollary}

\section{Maps on $M_n$}

 Corollary \ref{thm: trace preserver on special spaces}  shows that every linear bijection $\phi:M_n\to M_n$ corresponds to another linear bijection $\psi:M_n\to M_n$ such that
$\tr(\phi(A)\psi(B))=\tr(AB)$ for all $A, B\in M_n$.
  If there are three or more maps $M_n\to M_n$  preserve  trace of products,
they can be explicitly described as follow.

\begin{theorem}
\label{thm: M_n trace preserve m product}
Let $m,n\in\Z_+$ and $m\ge 3$. Maps $\phi_1,\ldots,\phi_m: M_n\to M_n$ satisfy that
\begin{equation}\label{eq M_n trace preserve m product}
\tr(\phi_1(A_1)\cdots\phi_m(A_m))=\tr(A_1\cdots A_m),\qquad \forall A_1,\ldots, A_m\in M_n,
\end{equation}
if and only if there are invertible matrices $N_1,\ldots,N_m\in M_n$ and $N_{m+1}:=N_1$ such that
\begin{equation}\label{formula M_n trace preserve m product}
\phi_i(A)=N_i A N_{i+1}^{-1}, \quad\forall A\in M_n.
\end{equation}
\end{theorem}

\begin{proof} It suffices to prove that \eqref{eq M_n trace preserve m product} implies \eqref{formula M_n trace preserve m product}.
Let $\psi_1(B):= \phi_2(B)\phi_3(I)\cdots\phi_m(I)$. Then $\tr(\phi_1(A)\psi_1(B))=\tr(AB)$ for $A, B\in M_n$.
By Corollary \ref{thm: trace preserver on special spaces}, $\phi_1$ and $\psi_1$ are linear bijections. So $\phi_3(I),\cdots, \phi_m(I)$ are invertible.
Similarly, all $\phi_i$'s are linear bijections and all $\phi_i(I)$'s are invertible.

For all $A, B_2,\ldots,B_m\in M_n$ we have
$$
\tr(\phi_1(A)\phi_2(B_2)\cdots\phi_m(B_m))=\tr(AB_2\cdots B_m)=\tr(\phi_1 (A)\psi_1(B_2\cdots B_m)).
$$
Since $\Ima \phi_1=M_n$, we get
\begin{equation}
\phi_2(B_2)\cdots\phi_m(B_m)=\psi_1(B_2\cdots B_m)=\phi_2(B_2\cdots B_m)\phi_3(I)\cdots\phi_m(I).
\end{equation}
Let $B_4=\cdots=B_m=I$ then
\begin{equation}\label{M_n product preserve trace 2 n 3}
\phi_2(B_2)\phi_3(B_3)=\phi_2(B_2B_3)\phi_3(I).
\end{equation}
Let $B_2=I$ in \eqref{M_n product preserve trace 2 n 3}, then
$\phi_2(I)\phi_3(B_3)=\phi_2(B_3)\phi_3(I)$ so that $\phi_3(B_3)=\phi_2(I)^{-1}\phi_2(B_3)\phi_3(I)$.
So \eqref{M_n product preserve trace 2 n 3} becomes
$\phi_2(B_2)\phi_2(I)^{-1}\phi_2(B_3)=\phi_2(B_2B_3).$  Let $\widehat{\phi_2}(A):=\phi_2(A)\phi_2(I)^{-1}$ then
$\widehat{\phi_2}(I)=I$ and
\begin{equation}
\widehat{\phi_2}(A)\widehat{\phi_2}(B)=\widehat{\phi_2}(AB),\quad \forall A, B\in M_n.
\end{equation}
So $\widehat{\phi_2}$ is an algebra automorphism of $M_n$, which must be of the form $\widehat{\phi_2}(A)=N_2AN_2^{-1}$ for an invertible $N_2\in M_n$. Let
$N_3:=\phi_2(I)^{-1}N_2$ and $N_4:=\phi_3(I)^{-1}N_3$. Then
\begin{eqnarray}
\phi_2(A) &=& \widehat{\phi_2}(A)\phi_2(I)=N_2 AN_3^{-1},
\\
\phi_3(A) &=& \phi_2(I)^{-1}\phi_2(A)\phi_3(I)= \phi_2(I)^{-1}N_2AN_3^{-1}\phi_3(I)=N_3 AN_4^{-1}.
\end{eqnarray}
Similarly, there are invertible $N_5,\ldots, N_m, N_{m+1}$ and $N_1:=N_{m+1}$ such that $\phi_i(A)=N_i A N_{i+1}^{-1}$ for all $i=1,\ldots,m$.
\end{proof}

\begin{remark}
Theorem \ref{thm: M_n trace preserve m product} also holds if all $M_n$ are replaced by $M_n(\R)$, and the proof is analogous.
\end{remark}

The following two examples show that two maps $M_n\to M_n$ preserving trace of product may not be of the form \eqref{formula M_n trace preserve m product}.

\begin{example}\label{Ex: Hadamard}
   Suppose that $\phi$ and $\psi$ are two maps on $M_n$ which are defined by $\phi(A) = A \circ C, \quad \psi (A)=A\circ \hat{C},$ where $C=(c_{ij})\in S_n(\mathbb{R})$ and $\hat{C}=(\frac{1}{c_{ij}})$ with $c_{ij}\neq0$ for all $i,j$. Here, $A\circ C$ denotes the Hadamard product of $A$ and $C$. Then the maps $\phi$ and $\psi$ send from $H_n$ to $H_n$ and $S_n$ to $S_n$. Moreover, one can show that $\phi$ and $\psi$ satisfy the equality $$\tr \phi(A)\psi(B)=\tr AB, \quad \forall A, B\in M_n.$$
\end{example}

\begin{example}\label{Ex: PQ}
  Let $A_1, \ldots, A_n, B_1, \ldots, B_n\in M_n$ be invertible matrices with $A_iB_i=I_n$ for all $i=1, \ldots, n$. Let $\{P_{ij}: 1\leq i, j\leq n \}$ and $\{Q_{ij}: 1\leq i, j\leq n \}$ be the subsets of $M_n$, where $P_{ij}=E_{ij}A_i$ and $Q_{ij}=B_jE_{ij}$. Define two maps $\phi$ and $\psi$ on $M_n$ by
  $$\phi(\sum_{i,j=1}^nc_{ij}E_{ij})=\sum_{i,j=1}^nc_{ij}P_{ij}, \quad \psi(\sum_{i,j=1}^nd_{ij}E_{ij})=\sum_{i,j=1}^nd_{ij}Q_{ij}.$$
   Then one can show that
   $$\tr \phi(A)\psi(B)=\tr AB, \quad \forall A, B\in M_n.$$
\end{example}

\subsection{Linear Maps $M_n\to M_k$}

Now let  $m\ge 2$ and we further consider linear maps $\phi_1,\ldots,\phi_m: M_n\to M_k$ preserving trace of products:
\be{equation}
{\label{M_n->M_k}
\tr(\phi_1(A_1)\phi_2(A_2)\cdots\phi_m(A_m))=\tr(A_1A_2\cdots A_m),\quad A_1,A_2,\ldots,A_m\in M_n.
}

\be{theorem}
{\label{thm: trace preserving n>k impossible}
When $n>k$ and $m\ge 2$, there are no linear maps $\phi_1,\ldots,\phi_m: M_n\to M_k$ that satisfy \eqref{M_n->M_k}.
}

\be{proof}
{
 Suppose on the contrary,  there exist linear maps $\phi_1,\ldots,\phi_m: M_n\to M_k$ that satisfy
 $\tr(\phi_1(A_1)\cdots\phi_m(A_m))=\tr(A_1\cdots A_m)$ for all $A_1,\ldots,A_m\in M_n.$
Let
$$\psi(B):=\phi_2(B)\phi_3(I_n)\cdots\phi_m(I_n).$$
Then both $\phi_1$ and $\psi$ are linear maps and $\tr(\phi_1(A)\psi(B))=\tr(AB)$ for $A, B\in M_n$.
Let $\{B_1, B_2, \ldots, B_{n^2}\}$ be a basis of $M_n$. The matrix
$\mtx{\tr(B_i B_j)}_{n^2\times  n^2}$ has full rank $n^2$, but the matrix
$\mtx{\tr(\phi_1(B_i) \psi(B_j) )}_{n^2\times  n^2}$ has rank no more than $k^2<n^2$. This is a contradiction.
}

Let $0_n$ denote the zero matrix  of size $n\times n$.

\be{theorem}
{\label{thm: M_n to M_k 2 maps}
Suppose $n\le k$ and linear maps $\phi_1, \phi_2: M_n\to M_k$  satisfy \eqref{M_n->M_k} for  $m=2$:
\be{equation}{ \label{M_n->M_k m=2}
\tr(\phi_1(A)\phi_2(B))=\tr(AB),\quad A, B\in M_n.
}
Then there are linear bijections $\wt{\phi}_1, \wt{\phi}_2: M_k\to M_k$ such that
\be{eqnarray}
{\label{M_k->M_k extension m=2}
\tr(\wt{\phi}_1(\wt{A}) \wt{\phi}_2(\wt{B}))&=& \tr(\wt{A}  \wt{B}),\quad \wt{A},  \wt{B} \in M_k,
\\
\wt{\phi}_i(A\oplus 0_{k-n}) &=& \phi_i(A), \quad A\in M_n,\ i=1,2.
}
Moreover, if $\phi_1$ and $\phi_2$ both preserve Hermitian matrices, then
$\wt{\phi}_1$ and $\wt{\phi}_2$ can be chosen  to preserve Hermitian matrices.
}

\begin{proof}   Denote
\be{eqnarray}
{
V&:=&\{A\oplus 0_{k-n}: A\in M_n\},\\
W&:=&\left\{\mtx{0_n &B\\C &D}_{k\times k}: B\in M_{n,k-n}, C\in M_{k-n,n}, D\in M_{k-n}\right\}.
}
Then $M_k=V\oplus W$ and $\tr(\wt{A}\wt{B})=0$ for all $\wt{A}\in V$ and $\wt{B}\in W$.
Let $\{\wt{C_1},\ldots,\wt{C_{k^2-n^2}}\}$ be a basis of $W$. The matrix
$\mtx{\tr(\wt{C_p}\wt{C_q})}_{(k^2-n^2)\times (k^2-n^2)}$ is invertible.

For $i=1,2,$ let
\be{equation}{
X_i :=  \left\{\wt{B}\in M_k: \tr(\phi_i(A)\wt{B})=0\ \text{for all } A\in M_n\right\}.
}
The \eqref{M_n->M_k m=2}  implies that  $\Ima(\phi_1)\cap X_2=\Ima(\phi_2)\cap X_1=\{0_k\}$.
Corollary \ref{thm: trace preserving injection} shows that $\phi_1$ and $\phi_2$ are injections. Hence
\be{equation}{
M_k=\Ima(\phi_1)\oplus X_2=\Ima(\phi_2)\oplus X_1.
}
Fix a basis $\{\wt{E_1},\ldots,\wt{E_{k^2-n^2}}\}$ of $X_1$.
There is a unique basis $\{\wt{F_1},\ldots,\wt{F_{k^2-n^2}}\}$ of $X_2$ such that
\be{equation}
{\label{tr EF=tr CC}
\mtx{\tr(\wt{E_i}\wt{F_j})}_{(k^2-n^2)\times (k^2-n^2)}=\mtx{\tr(\wt{C_i}\wt{C_j})}_{(k^2-n^2)\times (k^2-n^2)}.
}
Let $\wt{\phi}_1, \wt{\phi}_2:M_k\to M_k$ be the linear maps defined by
\begin{itemize}
\item $\wt{\phi}_i(A\oplus 0_{k-n})=\phi_i(A)$ for $A\in M_n$ and $i=1,2$;
\item $\wt{\phi}_1(\wt{C_j})=\wt{E_j}$ and  $\wt{\phi}_2(\wt{C_j})=\wt{F_j}$ for $j=1,\ldots,k^2-n^2$.
\end{itemize}
Then $\wt{\phi}_1$ and $\wt{\phi}_2$ are linear bijections,   $\wt{\phi}_i(A\oplus 0_{k-n})=\phi_i(A)$  for
  $A\in M_n$ and $i=1,2,$ and \eqref{M_k->M_k extension m=2} holds.

Finally, suppose $\phi_1$ and $\phi_2$ both preserve Hermitian matrices.  Since $M_n$ has a basis consisting of Hermitian matrices, so do
$\Ima(\phi_1)$ and $\Ima(\phi_2)$. We can let the sets
$\{\wt{C_1},\ldots,\wt{C_{k^2-n^2}}\}$ and $\{\wt{E_1},\ldots,\wt{E_{k^2-n^2}}\}$ consist entirely of Hermitian matrices.
Then each $\wt{F_j}$ in the set $\{\wt{F_1},\ldots,\wt{F_{k^2-n^2}}\}$ satisfies that
$$\tr(\phi_2(A)\wt{F_j})=0\in\R,\quad \tr(\wt{E_i}\wt{F_j})=\tr(\wt{C_i}\wt{C_j})\in\R,
$$
which is possible only when each $\wt{F_j}$ is a Hermitian matrix.  Therefore, the maps
$\wt{\phi}_1$ and $\wt{\phi}_2$ constructed above preserve Hermitian matrices.
\end{proof}

\begin{example}\label{M_n-M_k m=3 counterexample}
Analogous linear map extension result of Theorem \ref{thm: M_n to M_k 2 maps} for $m>2$ does not hold. Consider $k=2n$ and $m=3$. Let $X\in M_n$ be an arbitrary nonscalar matrix. Let
$\phi_1, \phi_2, \phi_3: M_n\to M_{2n}$ such that
\be{equation*}
{
\phi_1(A)=A\oplus 0_n,\quad
\phi_2(B)=B\oplus B,\quad
\phi_3(C)=C\oplus XCX^*,\quad
A, B, C\in M_n.
}
Obviously,
$$
\tr(\phi_1(A)\phi_2(B)\phi_3(C))=\tr(ABC),\quad A, B, C\in M_n.
$$
Suppose there are linear bijections $\wt{\phi}_1, \wt{\phi}_2, \wt{\phi}_3:M_{2n}\to M_{2n}$
such that $\wt{\phi}_i(A\oplus 0_n)=\phi_i(A)$ for $A\in M_n$ and $i=1,2,3,$ and
$$
\tr(\wt{\phi}_1(\wt{A})\wt{\phi}_2(\wt{B})\wt{\phi}_3(\wt{C}))=\tr(\wt{A}\wt{B}\wt{C}),\quad \wt{A}, \wt{B}, \wt{C} \in M_{2n}.
$$
Then by \eqref{M_n product preserve trace 2 n 3}, it is necessary that for all $B, C\in M_n$,
\be{eqnarray}
{\notag
\wt{\phi}_2(B\oplus 0_n )\wt{\phi}_3(C\oplus 0_n  ) &=& \wt{\phi}_2(BC\oplus 0_n ) \wt{\phi}_3(I_{2n}),\quad\text{that is,}
\\ \label{M_n-M_k m=3 problem}
(B\oplus B)(C\oplus XCX^*)=BC\oplus BXCX^* &=& (BC\oplus BC) \wt{\phi}_3(I_{2n}).
}
However, since $X$ is selected from  nonscalar matrices, no matrix $\wt{\phi}_3(I_{2n})\in M_{2n}$
can satisfy \eqref{M_n-M_k m=3 problem} for all $B, C\in M_n.$
So the statement in Theorem \ref{thm: M_n to M_k 2 maps} is not true if $m>2$.
This example also indicates that analogous statement on maps $H_n\to H_k$  is not true for $m>2$ either.
\end{example}



\section{Maps on $H_n$ and $P_n$}

In this section, we investigate maps $H_n\to H_n$ and $P_n\to \ol{P_n}$ that preserve a trace of products property.  These two types of maps lead to different results, especially in the trace of products of two maps.


\subsection{Maps $H_n\to H_n$ preserving trace of products}

By Corollary \ref{thm: trace preserver on special spaces}, every linear bijection $\phi: H_n\to H_n$
corresponds to another linear bijection $\psi: H_n\to H_n$ such that
 $\tr(\phi(A)\psi(B))=\tr(AB)$ for all $A, B\in H_n$.
Some examples can be analogously developed from Example \ref{Ex: Hadamard}.

We  investigate the problem of three or more maps $H_n\to H_n$ preserving trace of products.

\begin{theorem}\label{thm: H_n trace preserve m product}
Suppose $m\ge 3$. Maps $\phi_1,\ldots,\phi_m: H_n\to H_n$ satisfy that
\begin{equation}\label{eq H_n trace preserve m product}
\tr(\phi_1(A_1)\cdots\phi_m(A_m))=\tr(A_1\cdots A_m),\quad \forall A_1,\ldots,A_m\in H_n,
\end{equation}
if and only if  they can be described as follow:
\begin{enumerate}
\item When $m$ is odd, there exist a unitary matrix $U\in M_n$ and nonzero real scalars $c_1,\ldots, c_m$ such that $c_1\cdots c_m=1$ and
\begin{equation}\label{H_n trace preserving multiple products odd}
\phi_i(A)=c_i U^*AU,\quad \forall A\in H_n,\ i=1,\ldots,m.
\end{equation}

\item When $m$ is even, there exist an invertible $M\in M_n$ and nonzero real scalars $c_1,\ldots, c_m$ such that $c_1\cdots c_m=1$ and
\begin{equation}\label{H_n trace preserving multiple products even}
\phi_i(A)=
\begin{cases}
c_i M^*AM, &\text{$i$ is odd,}
\\
c_i M^{-1}AM^{-*}, &\text{$i$ is even,}
\end{cases}
\quad \forall A\in H_n.
\end{equation}
\end{enumerate}
\end{theorem}

\begin{proof} It suffices to prove that \eqref{eq H_n trace preserve m product} implies \eqref{H_n trace preserving multiple products odd} (when $m$ is odd) or \eqref{H_n trace preserving multiple products even}  (when $m$ is even).

Let $\psi_1: H_n\to M_n$ be defined by $$\psi_1(B)=\phi_2(B)\phi_3(I)\cdots \phi_m(I),\qquad B\in H_n.$$
The \eqref{eq H_n trace preserve m product} implies that
$\tr(\phi_1(A)\psi_1(B))=\tr(AB)$ for $A,B\in H_n.$ By Theorem \ref{thm: two maps preserving trace} and the fact that the complex span of $H_n$ is $M_n$, $\phi_1$ can be extended to a linear bijection $\wt{\phi}_1$ on $M_n$. Similarly, each $\phi_i$ $(1\le i\le m)$ can be extended to a linear bijection $\wt{\phi}_i$ on $M_n$.

Now both sides of \eqref{eq H_n trace preserve m product} can be viewed as  multilinear functions on $(A_1,\ldots,A_m)\in {H_n}^{\times m}$.
Therefore,  $\wt{\phi}_1,\ldots, \wt{\phi}_m$ satisfy equality \eqref{eq M_n trace preserve m product}
after replacing each $\phi_i$ with $\wt{\phi}_i$.  Theorem \ref{thm: M_n trace preserve m product} says that there are invertible $N_1,\ldots, N_m\in M_n$ and $N_{m+1}:=N_1$ such that
each $\wt{\phi}_i (A)=N_iA N_{i+1}^{-1}$ for $A\in M_n$.  So for all $A\in H_n$, $\phi_i(A)=N_iA N_{i+1}^{-1}\in H_n$.  This could happen if and only if
$N_{i+1}^{-1}=c_i' N_i^{*}$ for a nonzero real scalar $c_i'$, in which $\phi_i(A)=c_i' N_iAN_i^*$. Let $M:=N_1^*$. Then
$\phi_1(A)=c_1 M^*AM$ for a nonzero real number $c_1$, and $\phi_2(A)=c_2 M^{-1}AM^{-*}$, $\phi_3(A)=c_3 M^*AM$, and so on.
\begin{enumerate}
\item When $m$ is odd, $\phi_m (A)=c_m M^*AM=N_m A N_{m+1}^{-1}=N_m A N_1^{-1}=N_m A M^{-*}.$ So $M=cM^{-*}$ or $MM^*=cI$ for a scalar $c$. By adjusting the scalar, we may assume that $U:=M$ is unitary, and $\phi_i$ are given by \eqref{H_n trace preserving multiple products odd}. Then necessarily $c_1\cdots c_m=1$.
\item When $m$ is even, we get the form \eqref{H_n trace preserving multiple products even}. Clearly  $c_1\cdots c_m=1$.
\qedhere
\end{enumerate}
\end{proof}


\begin{remark}
Real Hermitian matrices are  real symmetric matrices.
The real version of Theorem \ref{thm: H_n trace preserve m product} also holds.  See Theorem \ref{thm: S_n(R) trace preserve m product} on $S_n(\R)$.
\end{remark}

\subsection{Maps $P_n\to \ol{P_n}$ preserving trace of products}

Now consider the maps $P_n\to\ol{P_n}$ that preserve trace of products.
Many properties  of these maps can be generalized to their counterparts for the maps $P_n\to P_n$
and $\ol{P_n}\to\ol{P_n}$.  Note that linear maps $\ol{P_n}\to\ol{P_n}$ are positive maps,
which are popular and useful in quantum information theory.

Unlike the situations in $M_n\to M_n$ and $H_n\to H_n$, two maps
$\phi,\psi: P_n\to\ol{P_n}$ (similarly for $P_n\to P_n$ and $\ol{P_n}\to\ol{P_n}$) preserving trace of products have the following special forms.

\begin{theorem}\label{thm: P_n preserve trace}
Two maps $\phi:P_n\to \overline{P_n}$ and $\psi:P_n\to \overline{P_n}$  satisfy
\begin{equation}\label{trace preserver P_n}
\tr (\phi(A)\psi(B))=\tr(AB),\quad \forall A, B\in P_n,
\end{equation}
if and only if the following claims hold:
\begin{enumerate}
\item If   $\phi(I)=I$, then there exists a unitary matrix $U\in M_n$ such that
\begin{equation}\label{Pn to Pn preserve trace and I}
\phi(A)=\psi(A)=U^* AU \quad\text{or}\quad \phi(A)=\psi(A)=U^* A^t U,\quad \forall A\in P_n.
\end{equation}

\item In general, there exists an invertible $M\in M_n$ such that for all $A\in P_n$:
\begin{subequations}\label{Pn to Pn preserve trace}
\begin{eqnarray}
&\phi(A)=M^*AM,\  \psi(A)=M^{-1} AM^{-*};&\ \text{or \ }
\\
&\phi(A)=M^*A^tM,\ \psi(A)=M^{-1} A^tM^{-*}.&
\end{eqnarray}
\end{subequations}
\end{enumerate}
\end{theorem}

\begin{proof} The maps defined by \eqref{Pn to Pn preserve trace and I} and \eqref{Pn to Pn preserve trace} obviously satisfy
 \eqref{trace preserver P_n}. We shall prove the converse.
By Theorem \ref{thm: two maps preserving trace}, $\phi$ and $\psi$ can be extended to linear bijections $H_n\to H_n$, still denoted as $\phi$ and $\psi$ here,
that satisfy \eqref{trace preserver P_n}.

\begin{enumerate}
\item If  $\phi(I)=I$, then \eqref{trace preserver P_n} implies that $\tr \psi(A)=\tr A$ for $A\in P_n$.
Given a rank one projection $P$, we have $\phi(P),\psi(P)\in\overline{P_n}$ so that
\begin{eqnarray}
1&=&\tr P=\tr(P^2)=\tr(\phi(P)\psi(P))
\label{trace preserve rank 1 a}
\\
&\le&  [\tr (\phi(P)^2)\tr (\psi(P)^2)]^{1/2}
\label{trace preserve rank 1 b}
\\
&\le& \tr \phi(P)\tr \psi(P)=\tr \phi(P)\tr P=\tr \phi(P).
\label{trace preserve rank 1 c}
\end{eqnarray}
Let $E_{ij}\in M_n$ be the matrix with a sole nonzero value $1$ in the $(i,j)$ entry.
Write $P=WE_{11}W^*$ for a unitary $W$. Then $I$ is a sum of $n$ rank one projections
$I=P+P^{(2)}+\cdots+P^{(n)}$ in which $P^{(i)}:=WE_{ii}W^*$ for $2\le i\le n$.
\begin{equation}
n=\tr \phi(I)=\tr \phi(P)+\tr \phi(P^{(2)})+\cdots+\tr\phi(P^{(n)})\ge n.
\end{equation}
Therefore, $\tr\phi(P)=1$. The Cauchy-Schwartz inequality  \eqref{trace preserve rank 1 b} becomes equality  only if
$\psi(P)=c\phi(P)$ for some $c\ge 0$; the inequality in  \eqref{trace preserve rank 1 c} becomes equality only if both $\phi(P)$ and $\psi(P)$ are rank one.
Then $\tr \phi(P)=1$ and $\tr \psi(P)=\tr P=1$  imply that $\psi(P)=\phi(P)$ are rank one projections.

Since all rank one projections span $H_n$, we have $\phi(A)=\psi(A)$ for all $A\in P_n$. Then  \eqref{trace preserver P_n} becomes
\be{equation}{
\tr(\phi(A)\phi(B))=\tr(AB),\qquad \forall A, B\in P_n.
}
By \cite[Lemma 3.6]{Huang16}, there is a unitary $U$ such that $\phi=\psi$ are given by \eqref{Pn to Pn preserve trace and I}:
$$
\phi(A)=\psi(A)=U^*AU\qquad\text{or}\qquad \phi(A)=\psi(A)=U^*A^tU.
$$

\item In general, we first show that $\phi(I)$ is invertible.
Suppose on the contrary, $\phi(I)$ is not invertible.   Then $\phi(I)v=0$ for certain nonzero $v\in\C^n$.
By the linear bijectivity of $\psi$ on $H_n$, there exists a nonzero   $B\in H_n$ such that $\psi(B)=vv^*$.
If $B\in\ol{P_n}$ then
\be{equation}{
0=\tr(\phi(I)vv^*)=\tr(\phi(I)\psi(B))=\tr(B),
}
which contradicts $B\ne 0$. So $B\not\in\ol{P_n}$ and $B$ has a negative eigenvalue. Using spectral decomposition, there exists $A\in P_n$ such that $\tr(AB)<0.$ Then $\phi(A)\in \ol{P_n}$ and
\be{equation}{
0>\tr(AB)=\tr(\phi(A)\psi(B))\ge 0,
}
which is also a contradiction. Therefore, $\phi(I)\in\ol{P_n}$ must be invertible.

Define $\phi_1(A):=\phi(I)^{-1/2}\phi(A)\phi(I)^{-1/2}$ and $\psi_1(B):=\phi(I)^{1/2}\psi(B) \phi(I)^{1/2}$. Then  $\phi_1:P_n\to \overline{P_n}$ and $\psi_1: P_n\to \overline{P_n}$   satisfy
\eqref{trace preserver P_n} and $\phi_1(I)=I$. By preceding argument,  there is a unitary $U\in M_n$ such that $\phi_1(A)=\psi_1(A)=U^*AU$ or
$\phi_1(A)=\psi_1(A)=U^*A^t U$.
Therefore, let $M:=U\phi(I)^{1/2}$,  then either $\phi(A)=M^* AM$ and $\psi(A)=M^{-1} AM^{-*}$, or
$\phi(A)=M^*A^t M$ and $\psi(A)=M^{-1}A^t M^{-*}$. 
\qedhere
\end{enumerate}
\end{proof}

Next consider $m\ge 3$ maps $\phi_1,\ldots,\phi_m$ from $P_n$ to $\ol{P_n}$ that preserve trace of products. The following result can be proved by
Theorem  \ref{thm: H_n trace preserve m product}.  We will give another proof  by Theorem \ref{thm: P_n preserve trace}.

\begin{theorem} \label{thm: multiple product trace preserver}
Let $m\ge 3$.
Maps $\phi_i: P_n\to \ol{P_n}$ ($i=1,\ldots,m$)  satisfy that
\begin{equation}\label{trace preserving multiple products}
\tr (\phi_1(A_1)\cdots \phi_m(A_m))=\tr (A_1\cdots A_m),\quad\forall A_1,\ldots,A_m\in P_n,
\end{equation}
if and only if they are described as follow:
\begin{enumerate}
\item When $m$ is odd, there exist a unitary matrix $U\in M_n$ and scalars $c_1,\ldots, c_m\in\R_+$ such that $c_1\cdots c_m=1$ and
\begin{equation}\label{trace preserving multiple products odd}
\phi_i(A)=c_i U^*AU,\quad \forall A\in P_n,\ i=1,\ldots,m.
\end{equation}

\item When $m$ is even, there exist an invertible $M\in M_n$ and scalars $c_1,\ldots, c_m\in\R_+$ such that $c_1\cdots c_m=1$ and
\begin{equation}\label{trace preserving multiple products even}
\phi_i(A)=
\begin{cases}
c_i M^*AM, &\text{$i$ is odd,}
\\
c_i M^{-1}AM^{-*}, &\text{$i$ is even,}
\end{cases}
\quad \forall A\in P_n.
\end{equation}
\end{enumerate}
\end{theorem}


\begin{proof}
It suffices to prove that \eqref{trace preserving multiple products} implies \eqref{trace preserving multiple products odd}  when $m$ is odd and \eqref{trace preserving multiple products even}  when $m$ is even.
 Let $C:=\phi_3(I)\cdots\phi_m(I)$.
In Theorem \ref{thm: two maps preserving trace}, let $\phi:P_n\to \ol{P_n}$ such that $\phi(A):=\phi_1(A)$, and $\psi: P_n\to \ol{P_n} C$ such that
$\psi(A):=\phi_2(A) C$.  The complex spanning space $\langle P_n\rangle=M_n$.
So $\phi$ and $\psi$ satisfy the assumptions of Theorem \ref{thm: two maps preserving trace}. It implies that $\phi_1$ and $\phi_2$ can be extended
to linear bijections $\wt{\phi_1} : M_n \to M_n$ and $\wt{\phi_2}: M_n\to M_n$ respectively, which necessarily map $H_n$ to $H_n$ bijectively.

Let $A_1, A_2\in P_n$ and $B_1\in H_n$. There exists $t\in\R_+$ such that $A_1+tB_1\in P_n$. \eqref{trace preserving multiple products} implies that
\begin{eqnarray*}
&&\tr(A_1A_2)+t\tr (\wt{\phi_1}(B_1)\phi_2(A_2)C)
\\
&=&
\tr (\phi_1 (A_1)\phi_2(A_2)C)+t\tr (\wt{\phi_1}(B_1)\phi_2(A_2)C)
\\
&=& \tr (\phi_1(A_1+tB_1)\phi_2(A_2)C)
\\
&=& \tr [(A_1+tB_1)A_2]\in\R.
\end{eqnarray*}
Therefore, $\tr (\wt{\phi_1}(B_1)\phi_2(A_2)C)\in\R$ for any Hermitian matrix $B_1$. The bijectivity of $\wt{\phi_1}|_{H_n}:H_n\to H_n$ implies that
$\tr(B\phi_2(A_2)C)\in\R$ for all $B\in H_n$.

Write $\phi_2(A_2)C=A'+{\mathbf i} A''$ for $A', A''\in H_n$. Then
$$
\tr (A''A')+{\mathbf i}\tr (A'' A'')=\tr(A''\phi_2(A_2)C)\in\R.
$$
So $\tr (A'' A'')=0$ and $A''=0$. Hence $\phi_2(A_2)C\in H_n$ for all $A_2\in P_n$.


Every Hermitian matrix is the difference of two positive definite matrices. Thus  for all $B\in H_n$ we have $\wt{\phi_2}(B)C\in H_n$.
The bijectivity of
$\wt{\phi_2}|_{H_n}:H_n\to H_n$ implies that $AC\in H_n$ for all $A\in H_n$.  This is true only when
$\phi_3(I)\cdots\phi_m(I)=C=cI$ for some $c\in\R$; and
 \eqref{trace preserving multiple products}  implies that $c\in\R_+$.
 Let $c_{1,2}:=\frac{1}{c}$. Then $c_{1,2}\in\R_+$ and
\eqref{trace preserving multiple products} becomes
\begin{equation}\label{trace multiple product with identities 2}
\tr (\phi_1(A_1) \phi_2(A_2))=c_{1,2}\tr (A_1A_2),\quad \forall A_1, A_2\in P_n.
\end{equation}

For each $i\in \{1,2,\ldots,m\}$, the identity \eqref{trace preserving multiple products} is equivalent to
\begin{equation}\label{trace preserving products}
\tr (\phi_i(A_i)\cdots \phi_m(A_m)\phi_1(A_1)\cdots \phi_{i-1}(A_{i-1}))=\tr (A_i\cdots A_mA_1\cdots A_{i-1})
\end{equation}
for all $A_1,\ldots,A_m\in P_n.$
Similarly, we can prove that
\begin{equation}
\phi_{i+2}(I)\cdots\phi_m(I)\phi_1(I)\cdots\phi_{i-1}(I)=\frac{1}{c_{i,i+1}} I
\end{equation}
for some  $c_{i,i+1}\in\R_+$. So \eqref{trace preserving products} implies that
\begin{equation}\label{trace preserving i, i+1 product}
\tr (\phi_i(A_i)\phi_{i+1}(A_{i+1}))=c_{i,i+1}\tr (A_iA_{i+1}),\quad\forall A_i, A_{i+1}\in P_n.
\end{equation}

Theorem \ref{thm: P_n preserve trace} and \eqref{trace multiple product with identities 2} indicate that
there exist an invertible $M\in M_n$ and $c_1, c_2\in\R_+$ such that 
\begin{eqnarray} \label{multiple product trace preserving 1}
&\phi_1(A)=c_1 M^*AM,\  \phi_2(A)=c_2 M^{-1} AM^{-*};&\ \text{or \ }
\\ \label{multiple product trace preserving 2}
&\phi_1(A)=c_1 M^*A^tM,\ \phi_2(A)=c_2  M^{-1} A^tM^{-*}.&
\end{eqnarray}

\begin{enumerate}
\item If \eqref{multiple product trace preserving 1} holds, then \eqref{trace preserving i, i+1 product} for $i=2$ implies that
\begin{eqnarray*}
\tr(A_2M^{-*}\phi_3(A_3)M^{-1})
&=& \tr(M^{-1}A_2M^{-*}\phi_3(A_3))
\\
&=& \frac{1}{c_2} \tr(\phi_2(A_2)\phi_3(A_3))=\frac{c_{2,3}}{c_2}\tr(A_2A_3)
\end{eqnarray*}
for all $A_2, A_3\in P_n$. So  $\phi_3(A)=c_3 M^*AM$ for some $c_3\in\R_+$. Similarly,
\begin{equation}
\phi_i(A)=\begin{cases} c_i M^*AM &\text{if \ } 2\nmid i,\\
c_i M^{-1}AM^{-*} &\text{if \ } 2\mid i,\end{cases}
\quad\text{for \ } i\in\{1,\ldots,m\}.
\end{equation}
If $m$ is even, then \eqref{trace preserving multiple products} implies that $c_1c_2\cdots c_m=1$; so \eqref{trace preserving multiple products even} holds.
If $m$ is odd,  then \eqref{trace preserving i, i+1 product} for $i=1$ and $i=m$ imply that there are $c_1, c_{m+1}\in\R_+$ such that
$c_1M^*AM=c_{m+1}M^{-1}AM^{-*}$ and thus
$$
c_1MM^*AMM^*=c_{m+1}A, \quad\forall A\in P_n.
$$
Hence $MM^*$ is a scalar matrix. We may adjust the matrix $M$ and the scalars $c_i$'s such that $U:=M$ is unitary, and $c_1\cdots c_m=1$.
So \eqref{trace preserving multiple products odd}  holds.

\item
If  \eqref{multiple product trace preserving 2} holds, then similarly we have
\begin{equation}\label{multiple product preserving false case}
\phi_i(A)=\begin{cases} c_i M^*A^tM &\text{if \ } 2\nmid i,\\
c_i M^{-1}A^t M^{-*} &\text{if \ } 2\mid i,\end{cases}
\quad\text{for \ } i\in\{1,\ldots,m\}.
\end{equation}
Moreover, $c_1\cdots c_m=1$, and $M$ is unitary when $m$ is odd. However, when $m\ge 3$,
none of the  expressions \eqref{multiple product preserving false case}    satisfies \eqref{trace preserving multiple products}.  \qedhere
\end{enumerate}
\end{proof}

Theorem \ref{thm: P_n preserve trace}  and Theorem \ref{thm: multiple product trace preserver} can be generalized to a result involving powers.

\begin{corollary} \label{thm: multiple weighted product trace preserver}
Let an integer $m\ge 2$ and nonzero real numbers $\alpha_1,\ldots,\alpha_m,\beta_1,\ldots,\beta_m$ be given.
Maps $\phi_i: P_n\to \ol{P_n}$ ($i=1,\ldots,m$) satisfy that
\begin{equation}\label{trace preserving multiple weighted products}
\tr \left(\phi_1(A_1)^{\alpha_1}\cdots \phi_m(A_m)^{\alpha_m}\right)=\tr \left(A_1^{\beta_1}\cdots A_m^{\beta_m}\right),\quad\forall A_1,\ldots,A_m\in P_n,
\end{equation}
if and only if they are described as follow:
\begin{enumerate}
\item When $m$ is odd, there exists a unitary matrix $U\in M_n$ and scalars $c_1,\ldots, c_m\in\R_+$ such that $c_1\cdots c_m=1$ and
\begin{equation}\label{trace preserving multiple weighted products odd}
\phi_i(A)=c_i^{1/\alpha_i} U^*A^{\beta_i/\alpha_i} U,\quad \forall A\in P_n,\ i=1,\ldots,m.
\end{equation}

\item When $m$ is even, there exists an invertible $M\in M_n$ and scalars $c_1,\ldots, c_m\in\R_+$ such that $c_1\cdots c_m=1$ and
\begin{equation}\label{trace preserving multiple weighted products even}
\phi_i(A)=
\begin{cases}
c_i^{1/\alpha_i} \left(M^*A^{\beta_i} M\right)^{1/\alpha_i}, &\text{$i$ is odd,}
\\
c_i^{1/\alpha_i} \left(M^{-1}A^{\beta_i} M^{-*}\right)^{1/\alpha_i} , &\text{$i$ is even,}
\end{cases}
\quad \forall A\in P_n.
\end{equation}
\end{enumerate}
\end{corollary}

\begin{proof} It suffices to prove that  \eqref{trace preserving multiple weighted products} implies \eqref{trace preserving multiple weighted products odd} (when $m$ is odd) or \eqref{trace preserving multiple weighted products even} (when $m$ is even).
Define $\psi_i:P_n\to \ol{P_n}$ such that $\psi_i(A):=\phi_i(A^{1/\beta_i})^{\alpha_i}$ for $i=1,\ldots,m$. Then
$$
\tr(\psi_1(A_1)\cdots\psi_m(A_m))=\tr(A_1\cdots A_m),\quad\forall A_1,\ldots,A_m\in P_n.
$$
Apply Theorem \ref{thm: P_n preserve trace} for $m=2$ and Theorem \ref{thm: multiple product trace preserver} for $m\ge 3$ to complete the proof.
\end{proof}

All theorems in this subsection have counterparts in the maps $P_n(\R)\to\ol{P_n(\R)}$.  The real versions of
Theorem \ref{thm: P_n preserve trace} and Theorem \ref{thm: multiple product trace preserver} are as follow.

\begin{theorem} \label{thm: multiple product trace preserver real}
Let $m\ge 2$.
Maps $\phi_i: P_n(\R)\to \ol{P_n(\R)}$ ($i=1,\ldots,m$)  satisfy that
\begin{equation}\label{trace preserving multiple products real}
\tr (\phi_1(A_1)\cdots \phi_m(A_m))=\tr (A_1\cdots A_m)
\end{equation}
for all $ A_1,\ldots,A_m\in P_n(\R),$ if and only if they are described as follow:
\begin{enumerate}
\item When $m$ is odd, there exist an orthogonal matrix $O\in M_n(\R)$ and scalars $c_1,\ldots, c_m\in\R_+$ such that $c_1\cdots c_m=1$ and
\begin{equation}\label{trace preserving multiple products odd real}
\phi_i(A)=c_i O^t AO,\quad \forall A\in P_n(\R),\ i=1,\ldots,m.
\end{equation}

\item When $m$ is even, there exist an invertible $M\in M_n(\R)$ and scalars $c_1,\ldots, c_m\in\R_+$ such that $c_1\cdots c_m=1$ and
\begin{equation}\label{trace preserving multiple products even real}
\phi_i(A)=
\begin{cases}
c_i M^t AM, &\text{$i$ is odd,}
\\
c_i M^{-1}AM^{-t}, &\text{$i$ is even,}
\end{cases}
\quad \forall A\in P_n(\R).
\end{equation}
\end{enumerate}
\end{theorem}

\begin{proof}
We extend each $\phi_i$ to a map $S_n(\R)\to S_n(\R)$ and still denote it $\phi_i$ by abuse of terminologies.
For  each $A\in S_n(\R)$, choose arbitrary $c\in\R$ such that $A+cI_n\in P_n(\R)$.  Define
\be{equation}{
\phi_i(A)=\phi_i(A+cI_n)-c\phi_i(I_n).
}
The equality \eqref{trace preserving multiple products real} guarantees that each $\phi_i: S_n(\R)\to S_n(\R)$ is well-defined,
and they satisfy \eqref{trace preserving multiple products real} for all $ A_1,\ldots,A_m\in S_n(\R).$

When $m\ge 3$, the statement can be proved by Theorem \ref{thm: S_n(R) trace preserve m product} in the next section.

When $m=2$, we have $\tr(\phi_1(A_1)\phi_2(A_2))=\tr(A_1A_2)$ for all $A_1, A_2\in P_n(\R)$.
Similar to the proof of Theorem \ref{thm: P_n preserve trace} (2), we see that $\phi_1 (I)$ is invertible.
  Replacing
$(\phi_1(A_1), \phi_2(A_2))$ by $(\phi_1(I)^{-1/2}\phi_1(A_1)\phi_1(I)^{-1/2}, \phi_1(I)^{1/2}\phi_2(A_2)\phi_1(I)^{1/2})$ if necessary, we may assume that
$\phi_1(I)=I.$ Similar to the proof of Theorem \ref{thm: P_n preserve trace} (1), we  get $\phi_1=\phi_2$ and $\tr (\phi_1(A))=\tr (A)$ for $A\in P_n(\R).$
Therefore,
\be{equation}{\label{P_n real m=2}
\tr(\phi_1(A)^2)=\tr(\phi_1(A)\phi_2(A))=\tr(A^2)=\tr(\phi_1(A^2)),\quad \forall A\in P_n(\R).
}
By Kadison’s inequality \cite[p. 495]{Kadison52}, $\phi_1(A)^2\le \phi_1(A^2)$. The equality \eqref{P_n real m=2} forces $\phi_1(A)^2= \phi_1(A^2)$ for all $A\in P_n(\R)$ and thus for all $A\in S_n(\R)$. Now extend $\phi_1$ to a linear map $\wt{\phi_1}: S_n\to S_n$ such that for all $A, B\in S_n(\R)$:
\be{equation}{
\wt{\phi_1}(A+\i B)=\phi_1(A)+\i \phi_1(B).
}
Then
\be{eqnarray}{
\wt{\phi_1}((A+\i B)^2) &=& \phi_1(A^2-B^2)+\i\phi_1((A+B)^2-A^2-B^2)
\\
&=& \phi_1(A)^2-\phi_1(B)^2+\i \left(\phi_1(A+B)^2-\phi_1(A)^2-\phi_1(B)^2\right)
\\
&=& \phi_1(A)^2-\phi_1(B)^2+\i \left(\phi_1(A)\phi_1(B)+\phi_1(B)\phi_1(A)\right)
\\
&=& \wt{\phi_1}(A+\i B)^2.
}
By \cite[Theorem 2]{ChanLim}, there is an orthogonal matrix $O=[o_{ij}]\in M_n$ such that   $\phi_1(A)=\phi_2(A)=O^t AO$ for
$A\in S_n(\R)$. Note that every $B\in M_n(\R)$ is a product of two real symmetric matrices, say $B=A_1A_2$ for $A_1, A_2\in S_n(\R)$,
so that
$$O^t BO=O^t A_1O O^t A_2O=\phi_1(A_1)\phi_2(A_2)\in M_n(\R).$$
The $(p,q)$ entry of $O^t E_{ij}O$
is $o_{ip}o_{jq}\in\R$  for all   $i,j,p,q\in\{1,\ldots,n\}.$ Hence the entries of $O$ are either all real or all pure imaginary.
Since $O$ is orthogonal, it must be real orthogonal.
The statement for $m=2$ has been proved.
\end{proof}

\section{Maps on $S_n$}

By Corollary \ref{thm: trace preserver on special spaces}, every linear bijection $\phi: S_n\to S_n$
corresponds to another linear bijection $\psi: S_n\to S_n$ such that
 $\tr(\phi(A)\psi(B))=\tr(AB)$ for all $A, B\in S_n$.
 Some examples can be analogously developed from Example \ref{Ex: Hadamard}.

  We now study the problem of three or more maps $S_n\to S_n$ preserving trace of products.

\begin{theorem}\label{thm: S_n trace preserve m product}
Suppose $m\ge 3$. Maps $\phi_1,\ldots,\phi_m: S_n\to S_n$ satisfy that
\begin{equation}\label{eq S_n trace preserve m product}
\tr(\phi_1(A_1)\cdots\phi_m(A_m))=\tr(A_1\cdots A_m),\quad \forall A_1,\ldots,A_m\in S_n,
\end{equation}
if and only if they are described as follow:
\begin{enumerate}
\item When $m$ is odd, there exist an orthogonal matrix $O\in M_n$ and nonzero complex scalars $c_1,\ldots, c_m$ such that $c_1\cdots c_m=1$ and
\begin{equation}\label{S_n trace preserving multiple products odd}
\phi_i(A)=c_i O^tAO,\quad \forall A\in S_n,\ i=1,\ldots,m.
\end{equation}

\item When $m$ is even, there exist an invertible $M\in M_n$ and nonzero complex scalars $c_1,\ldots, c_m$ such that $c_1\cdots c_m=1$ and
\begin{equation}\label{S_n trace preserving multiple products even}
\phi_i(A)=
\begin{cases}
c_i M^tAM, &\text{$i$ is odd,}
\\
c_i M^{-1}AM^{-t}, &\text{$i$ is even,}
\end{cases}
\quad \forall A\in S_n.
\end{equation}
\end{enumerate}
\end{theorem}

\begin{proof} It suffices to prove that  \eqref{eq S_n trace preserve m product} implies \eqref{S_n trace preserving multiple products odd} (when $m$ is odd) or \eqref{S_n trace preserving multiple products even} (when $m$ is even).
Let $C_1:=\phi_3(I)\cdots\phi_m(I)$, and $\psi_1:S_n\to S_n C_1$ such that $\psi_1(B):=\phi_2(B)C_1$ for $B\in S_n$. Then
$\tr(\phi_1(A)\psi_1(B))=\tr(AB)$ for all $A, B\in S_n.$
Theorem \ref{thm: two maps preserving trace}  implies that $\phi_1$ and $\psi_1$ are linear bijections. Similarly, each $\phi_i$'s for $i=1,\ldots,m$ is a linear bijection.
If $C_1$ is singular, then  $v^tC_1=0$ for some $v\in \C^n\setminus\{0\}$, so that $vv^t\in S_n\setminus\{0\}$ but
$$
\psi_1(\phi_2^{-1}(vv^t))=vv^tC_1=0.
$$
This contradicts the bijectivity of $\psi_1$. Hence $C_1$ is invertible. So each of $\phi_3(I),\ldots,\phi_m(I)$, as well as $\phi_1(I)$ and $\phi_2(I)$,
is invertible.

First we prove the theorem for the case $m=3$.

For all $A, B, C\in S_n$ we have
\begin{eqnarray*}
&&\tr\left[\phi_1(A)\left(\phi_2(B)\phi_3(C)-\phi_3(B)\phi_2(C)\right)\right]
\\
&=& \tr (\phi_1(A)\phi_2(B)\phi_3(C)) -\tr (\phi_2(C)\phi_3(B)\phi_1(A))^t
\\
&=& \tr (ABC)-\tr(ACB) =\tr (ABC)-\tr(ABC)^t=0.
\end{eqnarray*}
Since $\phi_1:S_n\to S_n$ is bijective, we get
\begin{equation}\label{L_BC}
\phi_2(B)\phi_3(C)-\phi_3(B)\phi_2(C)\in K_n.
\end{equation}
Set $B=\phi_2^{-1}(I)$ and denote $J:=\phi_3\circ\phi_2^{-1}(I)\in S_n$ in \eqref{L_BC}. Then
\be{equation}
{
\phi_3(C)-J\phi_2(C)\in K_n,\qquad\forall C\in S_n.
}
Therefore,
\be{equation}
{\label{phi_3 by phi_2}
\phi_3(C)=\frac{1}{2}(J\phi_2(C)+\phi_2(C)J),\qquad\forall C\in S_n.
}
Apply the expression of $\phi_3$ in \eqref{phi_3 by phi_2} to \eqref{L_BC}:
\be{eqnarray}
{
&&\frac{1}{2}\phi_2(B)(J\phi_2(C)+\phi_2(C)J)-\frac{1}{2}(J\phi_2(B)+\phi_2(B)J)\phi_2(C)
\notag\\
&=& \frac{1}{2}\phi_2(B)\phi_2(C)J-\frac{1}{2}J\phi_2(B)\phi_2(C)\in K_n.
}
By surjectivity of $\phi_2$, the product $\frac{\phi_2(B)}{2}\phi_2(C)$ of two symmetric matrices  can be any matrix in $M_n$ (cf. \cite[Theorem 3.2.3.2]{Horn}).
So $JA-AJ\in K_n$ for all $A\in M_n$, which is equivalent to
\be{equation}
{\label{JK_n}
J(A-A^t)\in K_n,\qquad\forall A\in M_n.
}
The matrix $A-A^t$ in \eqref{JK_n} goes through all matrices in $K_n$. So \eqref{JK_n} holds if and only if $J=cI$ for some $c\in\C\setminus\{0\}$.
Then \eqref{phi_3 by phi_2} implies that
\be{equation}
{
\phi_3(C)=c\phi_2(C),\qquad\forall C\in S_n.
}
Similarly, $\phi_1$ is a nonzero constant multiple of $\phi_2$. So there exists a linear bijection $\phi:S_n\to S_n$ such that
$\phi_i=c_i'\phi$ for $i=1,2,3$, $c_1'c_2'c_3'=1$, and
\be{equation}
{
\tr(\phi(A)\phi(B)\phi(C))=\tr(ABC),\qquad\forall A, B, C\in S_n.
}

If $A\in S_n$ satisfies $A^2=A$, then $\tr(\phi(A)^2\phi(B))=\tr(\phi(A)\phi(I)\phi(B))$ for all $B\in S_n$. So
\be{eqnarray}
{
\phi(A)^2-\phi(A)\phi(I) &\in&  K_n,\quad\text{which implies that}
\\
2\phi(A)^2 &=& \phi(A)\phi(I)+\phi(I)\phi(A),
\\
\phi(A)^2\phi(I)+\phi(A)\phi(I)\phi(A) &=& 2\phi(A)^3=\phi(A)\phi(I)\phi(A)+\phi(I)\phi(A)^2,
\\
\phi(A)^2\phi(I) &=& \phi(I)\phi(A)^2.
\label{phi(A)^2 I}
}

Fix a basis $\{A_i: i=1,\ldots, \binom{n+1}{2},\ A_i^2=A_i\}$ of $S_n$. We claim that $\{\phi(A_i)^2: i=1,\ldots, \binom{n+1}{2}\}$ is also a basis of $S_n$.
Suppose $a_1,\ldots, a_{\binom{n+1}{2}}\in\C$ and
\be{equation}
{
a_1\phi(A_1)^2+\cdots+a_{\binom{n+1}{2}}\phi(A_{\binom{n+1}{2}})^2=0.
}
Then for all $C\in S_n$:
\be{eqnarray}
{
0 &=& \tr \left(a_1\phi(A_1)^2+\cdots+a_{\binom{n+1}{2}}\phi(A_{\binom{n+1}{2}})^2\right) \phi(C)
\notag\\ &=& \tr\left(a_1 A_1^2+\cdots+a_{\binom{n+1}{2}}A_{\binom{n+1}{2}}^2\right) C
\notag\\ &=& \tr\left(a_1 A_1+\cdots+a_{\binom{n+1}{2}}A_{\binom{n+1}{2}}\right) C.
}
Therefore, $a_1 A_1+\cdots+a_{\binom{n+1}{2}}A_{\binom{n+1}{2}}=0$, so that $a_1=\cdots=a_{\binom{n+1}{2}}=0$.

By \eqref{phi(A)^2 I},
\be{equation}
{
\phi(A_i)^2\phi(I) = \phi(I)\phi(A_i)^2,\qquad i=1,\ldots, \binom{n+1}{2}.
}
Hence $\phi(I)$ commutes with all matrices in $S_n$.
Therefore, $\phi(I)=c_0I$ for some $c_0\in\C.$ By
\be{equation}
{
c_0^3n=\tr(\phi(I)^3)=\tr(I^3)=n
}
we get $c_0^3=1$.

Denote the linear bijection $\wt{\phi}=\frac{1}{c_0}\phi$. Then $\wt{\phi}(I)=I$ and
\be{equation}
{
\tr(\wt{\phi}(A)\wt{\phi}(B)\wt{\phi}(C))=\tr(ABC),\qquad\forall A, B, C\in S_n.
}
Given $B\in S_n$, for all  $A\in S_n$,
\be{equation}
{
\tr(\wt{\phi}(A)\wt{\phi}(B^2))=\tr(\wt{\phi}(A)\wt{\phi}(B^2)\wt{\phi}(I))=\tr(AB^2)=\tr(\wt{\phi}(A)\wt{\phi}(B)^2).
}
So the symmetric matrices $\wt{\phi}(B^2)=\wt{\phi}(B)^2$. By \cite[Theorem 2]{ChanLim}, there exists an orthogonal matrix $O\in M_n$ such that $\wt{\phi}(A)=O^tAO$ for $A\in S_n$. Let $c_i=c_0c_i'$ for $i=1,2,3$.
Then we get \eqref{S_n trace preserving multiple products odd} for $m=3$.

From now on, we assume that $m\ge 4$.

The matrix $\phi_3(A_3)\cdots\phi_m(A_m)$ for $A_3,\ldots, A_m\in S_n$ can be any matrix in $M_n$ (cf. \cite[Theorem 3.2.3.2]{Horn}).
For all $A, A_3,\ldots, A_m\in S_n$ we have
\be{equation}
{
\tr\left[\left(\phi_1(A)\phi_2(I)-\phi_1(I)\phi_2(A)\right)\phi_3(A_3)\cdots\phi_m(A_m)\right]=0.
}
So $\phi_1(A)\phi_2(I) = \phi_1(I)\phi_2(A)$ for all $A\in S_n$. Hence
\be{equation}
{\label{tr multiple product phi_2 phi_1}
\phi_2(A) = \phi_1(I)^{-1}\phi_1(A)\phi_2(I).
}
Since $\phi_1(A)$ and $\phi_2(A)$ are symmetric and $\phi_1$ is surjective, there exists $c_2'\in\C\setminus\{0\}$ such that
\be{eqnarray}
{
\phi_2(I) &=& c_2' \phi_1(I)^{-t}= c_2' \phi_1(I)^{-1},\quad\text{so that by \eqref{tr multiple product phi_2 phi_1},}
\\
\label{tr multiple product phi_2}
\phi_2(A) &=& c_2' \phi_1(I)^{-1}\phi_1(A)\phi_1(I)^{-1},\qquad\forall A\in S_n.
}
Similarly, there is $c_3'\in\C\setminus\{0\}$ such that
\be{equation}
{
\phi_3(A) = c_3' \phi_2(I)^{-1}\phi_2(A)\phi_2(I)^{-1}=\frac{c_3'}{c_2'}\phi_1(A),\qquad\forall A\in S_n.
}
Therefore, the maps $\phi_1, \phi_3, \phi_5, \ldots$ (resp. $\phi_2, \phi_4, \phi_6, \ldots$) differ only by nonzero scalar multiples.
\be{enumerate}
{
\item If $m$ is odd, then $\phi_1=a_m\phi_m$  for some $a_m \in\C\setminus\{0\}$ and $\phi_m=a_2\phi_2$ for some $a_2 \in\C\setminus\{0\}$.
Therefore,   $\phi_1, \phi_2, \phi_3,\ldots,\phi_m$ differ only by nonzero scalar multiples.
 \eqref{tr multiple product phi_2 phi_1} implies that $\phi_1 (A)=\phi_1 (I)^{-1}\phi_1 (A)\phi_1 (I)$ for all $A\in S_n$, so that
$\phi_1 (I)=c_1 I$ for some $c_1\in\C\setminus\{0\}$. Let $\phi(A):=\frac{1}{c_1}\phi_1(A)$. Then $\phi(I)=I$ and $\phi_1= c_1\phi$. There exist  $c_2,\ldots,c_m\in \C\setminus\{0\}$ such that
$\phi_i=c_i\phi$ for $i=2,\ldots,m$. We have
$$c_1c_2\cdots c_m\tr(I^m)=c_1c_2\cdots c_m\tr(\phi(I)^m)=\tr(\phi_1(I)\cdots\phi_m(I))=\tr(I^m).$$
Therefore, $c_1\cdots c_m=1$, and
\be{equation}
{\label{tr multiple product one map 1}
\tr(\phi(A_1)\cdots\phi(A_m))=\tr(A_1\cdots A_m),\qquad\forall A_1,\ldots, A_m\in S_n.
}

\item If $m$ is even,
write $\phi_1(I)=M^tM$ for an invertible $M\in M_n$, then \eqref{tr multiple product phi_2} implies that
\be{equation}
{
M\phi_2(A)M^t=c_2' M^{-t}\phi_1(A)M^{-1},\qquad\forall A\in S_n.
}
There exist  a map $\phi(A)=  M^{-t}\phi_1(A)M^{-1}$ and scalars $c_1=1,c_2,\ldots,c_m\in\C\setminus\{0\}$   such that
\be{enumerate}
{
\item If $i\in\{1,\ldots,m\}$ is odd, then  $M^{-t}\phi_i(A)M^{-1}=c_i\phi(A)$ for $A\in S_n$.

\item If $i\in\{1,\ldots,m\}$ is even, then  $M\phi_i(A)M^t=c_i\phi(A)$ for $A\in S_n$.

\item $c_1\cdots c_m=1$, $\phi(I)=I$, and \eqref{tr multiple product one map 1} holds:
\be{equation*}
{
\tr(\phi(A_1)\cdots\phi(A_m))=\tr(A_1\cdots A_m),\qquad\forall A_1,\ldots, A_m\in S_n.
}
}
}

In both $m$ even and $m$ odd cases, we successfully find a linear bijection $\phi: S_n\to S_n$ that satisfies
$\phi(I)=I$ and
\eqref{tr multiple product one map 1}.
For all $A, A_3,\ldots, A_m\in S_n$ and $k\in\Z_+$ we have
\be{equation}
{
\tr\left[\left(\phi(A)\phi(A^{k-1})-\phi(I)\phi(A^{k})\right)\phi(A_3)\cdots\phi(A_m)\right]=0.
}
Hence $\phi(A)\phi(A^{k-1})=\phi(I)\phi(A^{k})=\phi(A^{k})$ and $\phi(A)^k=\phi(A^k)$ for $k\in\Z_+$.
Therefore, the following identities hold for every $A, B\in S_n$:
\be{equation}
{
\tr(\phi(A)\phi(B)^k)=\tr(\phi(A)\phi(B^k)\phi(I)\cdots\phi(I))=\tr(AB^k),\quad  \forall k\in\Z_+.
}
By the proof of $(3) \Rightarrow (1)$ in  \cite[Lemma 4.2]{Huang16}  or by   \cite[Theorem 2]{ChanLim},
there exists an orthogonal matrix $O\in M_n$ such that $\phi(A)=O^tAO$ for $A\in S_n$.
When $m$ is odd, we get  \eqref{S_n trace preserving multiple products odd}.
When $m$ is even, we get \eqref{S_n trace preserving multiple products even}.
\end{proof}

The real version of Theorem \ref{thm: S_n trace preserve m product}  holds. We state and prove it below.

\begin{theorem}\label{thm: S_n(R) trace preserve m product}
Suppose $m\ge 3$. Maps $\phi_1,\ldots,\phi_m: S_n(\R)\to S_n(\R)$ satisfy that
\begin{equation}\label{eq S_n(R) trace preserve m product}
\tr(\phi_1(A_1)\cdots\phi_m(A_m))=\tr(A_1\cdots A_m),\quad \forall A_1,\ldots,A_m\in S_n(\R),
\end{equation}
if and only if the following are true:
\begin{enumerate}
\item When $m$ is odd, there exist a real orthogonal matrix $O\in M_n(\R)$ and nonzero real scalars $c_1,\ldots, c_m$ such that $c_1\cdots c_m=1$ and
\begin{equation}\label{S_n(R) trace preserving multiple products odd}
\phi_i(A)=c_i O^tAO,\quad \forall A\in S_n(\R),\ i=1,\ldots,m.
\end{equation}

\item When $m$ is even, there exist an invertible $M\in M_n(\R)$ and nonzero real scalars $c_1,\ldots, c_m$ such that $c_1\cdots c_m=1$ and
\begin{equation}\label{S_n(R) trace preserving multiple products even}
\phi_i(A)=
\begin{cases}
c_i M^tAM, &\text{$i$ is odd,}
\\
c_i M^{-1}AM^{-t}, &\text{$i$ is even,}
\end{cases}
\quad \forall A\in S_n(\R).
\end{equation}
\end{enumerate}
\end{theorem}

\begin{proof} For every $A\in S_n$, denote $\Re A, \Ima A\in S_n(\R)$ such that $A=\Re A+\i \Ima A.$
For each $i=1,\ldots,m$, define $\wt{\phi_i}:S_n\to S_n$ in the way that
\be{equation}{
\wt{\phi_i}(A)=\phi_i(\Re A)+\i\phi_i(\Ima A),\quad\forall A\in S_n.
}
Then
\be{equation}{
\tr(\wt{\phi_1}(A_1)\cdots\wt{\phi_m}(A_m))=\tr(A_1\cdots A_m),\quad \forall A_1,\ldots,A_m\in S_n.
}
By Theorem \ref{thm: S_n trace preserve m product}:
\begin{enumerate}
\item When $m$ is odd, there exist an orthogonal matrix $O\in M_n$ and   $c_1,\ldots, c_m\in\C$ such that $c_1\cdots c_m=1$ and
\begin{equation}
\wt{\phi_i} (A)=c_i O^tAO,\quad \forall A\in S_n,\ i=1,\ldots,m.
\end{equation}
Then $\wt{\phi_i}(I)=c_i I\in S_n(\R)$ implies that each $c_i\in\R$. Write $O=[o_{ij}]_{n\times n}$. The $(p,q)$ entry of
 $\wt{\phi_1}(E_{ij})$ is $c_1 o_{ip}o_{jq}$, which must be real for all   $i,j,p,q\in\{1,\ldots,n\}.$ Hence the entries of $O$ are either all real or all pure imaginary.
Since $O$ is orthogonal, it has to be real orthogonal.

\item When $m$ is even, there exist an invertible $M\in M_n$ and  $c_1,\ldots, c_m\in\C$ such that $c_1\cdots c_m=1$ and
\begin{equation}
\wt{\phi_i} (A)=
\begin{cases}
c_i M^tAM, &\text{$i$ is odd,}
\\
c_i M^{-1}AM^{-t}, &\text{$i$ is even,}
\end{cases}
\quad \forall A\in S_n.
\end{equation}
We may assume that $c_1\in\R$ by adjusting $M$ with a scalar factor if necessary. Then $\wt{\phi_1} (I)\wt{\phi_2} (I)=c_1 c_2 I$ is real, which implies $c_2\in\R$. Similarly,  $c_i\in\R$ for all $i=1,\ldots,n.$ Using the fact that $\wt{\phi_1}(E_{jj})$ is real for all $j=1,\ldots,n$,
we see that the entries of $M$ are either all real or all pure imaginary.
We can assure that $M\in M_n(\R)$ by adjusting the signs of all $c_i$'s if necessary.

\end{enumerate}

Overall, the theorem is proved.
\end{proof}

\section{Maps on $D_n$}

The results about maps $D_n\to D_n$ in this section can be applied to maps between the spaces of upper (resp. lower)  triangular matrices with respective properties.
We view
\be{equation}
{\diag: \C^n\to D_n
}
 as the linear bijection  that sends $(c_1,\cdots,c_n)^t$ to the diagonal matrix $\diag(c_1,\ldots,c_n)$, and
\be{equation}
{\diag^{-1}:D_n\to \C^n
}
 the inverse map of $\diag$.

Let $\phi_1, \phi_2: D_n\to D_n$ satisfy that
\be{equation}
{\label{D_n trace preserver 2 product}
\tr(\phi_1(A_1)\phi_2(A_2))=\tr(A_1A_2),\quad\forall A_1, A_2\in D_n.
}
Then Theorem \ref{thm: two maps preserving trace} and Corollary \ref{thm: trace preserver on special spaces}  imply that
$\phi_1$ (resp. $\phi_2$) can be any linear automorphism, and $\phi_1$ and $\phi_2$ uniquely determine each other.
The following theorem gives the explicit description of these $(\phi_1,\phi_2)$ pairs.

\be{theorem}
{\label{thm: trace preserve 2 product}
Two maps $\phi_1, \phi_2: D_n\to D_n$ satisfy \eqref{D_n trace preserver 2 product} if and only if
there is  an invertible matrix $N\in M_n$ such that
\be{equation}
{
\phi_1(A)=\diag(N\diag^{-1}(A)),\quad \phi_2(A)=\diag(N^{-t}\diag^{-1}(A)),\quad\forall A\in D_n.
}
}

\be{proof}
{
For two vector $u, v\in\C^n$,
\be{equation}
{
\tr(\diag(u)\diag(v))=u^t v.
}
Define the  linear automorphisms $\wt{\phi}_i:\C^n\to\C^n$ by
\be{equation}
{
\wt{\phi}_i:=\diag^{-1}\circ\phi_i\circ\diag,\qquad i=1,2.
}
There exist invertible matrices $N, M\in M_n$ such that $\wt{\phi}_1(v)=Nv$ and $\wt{\phi}_2(v)=Mv$ for $v\in \C^n$.
Express elements of $D_n$ as $\diag(v)$ for  $v\in\C^n$. Then \eqref{D_n trace preserver 2 product} becomes
\be{equation}
{\label{two vectors preserve scalar product}
\wt{\phi}_1(u)^t \wt{\phi}_2(v)=u^t v,\quad\forall u, v\in\C^n.
}
It implies that $M=N^{-t}.$ Therefore, for $A\in D_n$,
\be{equation}
{
\phi_1(A)=\diag\circ\wt{\phi}_1\circ\diag^{-1}(A)=\diag(N\diag^{-1}(A))
}
and similarly $\phi_2(A)=\diag(N^{-t}\diag^{-1}(A)).$
}

\be{theorem}
{\label{thm: D_n trace preserve m product}
Suppose $m\ge 3$ and $\phi_1,\ldots,\phi_m: D_n\to D_n$ satisfy that
\begin{equation}\label{eq D_n trace preserve m product}
\tr(\phi_1(A_1)\cdots\phi_m(A_m))=\tr(A_1\cdots A_m),\quad \forall A_1,\ldots,A_m\in D_n.
\end{equation}
Then there exist  a permutation matrix $P\in M_n$ and invertible diagonal matrices $C_1,\ldots,C_m\in D_n$ with $C_1\cdots C_m=I$ such that
\be{equation}
{
\phi_i (A)=C_i P^t AP,\qquad\forall A\in D_n,\  i=1,\ldots,m.
}
}

\be{proof}
{
Let $C=\phi_3(I)\cdots\phi_m(I)$ and $\psi_1:D_n\to D_n$ such that $\psi_1(A)=\phi_2(A)C$. Then
$\tr(\phi_1(A)\psi_1(B))=\tr(AB)$ for all $A, B\in D_n$, which implies that both $\phi_1$ and $\psi_1$ are linear bijections.
So each of $\phi_3(I),\ldots,\phi_m(I)$  is invertible. Similarly, each of $\phi_1(I)$ and $\phi_2(I)$ is invertible, and each of
$\phi_2,\ldots,\phi_m$ is a linear bijection.

If $AB=A'B'$ for $A, B, A', B'\in D_n$, then for all $A_3,\ldots, A_m\in D_n$,
\be{eqnarray}
{\notag
&&\tr (\phi_1(A)\phi_2(B)\phi_3(A_3)\cdots \phi_m(A_m))
\\
&=& \tr(ABA_3\cdots A_m)=\tr(A'B'A_3\cdots A_m)
\\
&=& \tr (\phi_1(A')\phi_2(B')\phi_3(A_3)\cdots \phi_m(A_m)).
}
The bijectivities of $\phi_3,\ldots,\phi_m$ imply that $\phi_1(A)\phi_2(B)=\phi_1(A')\phi_2(B')$.  In particular,
\be{equation}
{\label{diagonal phi2 by phi1}
\phi_2(B)=\phi_1(I)^{-1}\phi_1(B)\phi_2(I),\qquad\forall B\in D_n.
}
For all $A, B\in D_n$ we also have
\be{equation}
{
\phi_1(AB)\phi_2(I) = \phi_1(A)\phi_2(B)=\phi_1(A) \phi_1(I)^{-1}\phi_1(B)\phi_2(I).
}
Let $\wt{\phi}_1(A):=\phi_1(A)\phi_1(I)^{-1}$ then
\be{equation}
{
\wt{\phi}_1(A)\wt{\phi}_1(B)=\wt{\phi}_1(AB),\quad \forall A, B\in D_n.
}
So $\wt{\phi}_1$ is an algebra automorphism of $D_n$, and the restriction of $\wt{\phi}_1$ on the group of invertible diagonal matrices in $D_n$ gives a group automorphism.
Hence there is a permutation matrix $P\in M_n$ such that
\be{equation}
{
\wt{\phi}_1(A)=P^t A P,\quad\forall A\in D_n.
}
Let $C_i:=\phi_i(I)$ for $i=1,\ldots,m$. Then $\phi_1(A)=C_1 P^t AP$, and by \eqref{diagonal phi2 by phi1},
\be{equation}{
\phi_2(A)=(P^t AP) C_2=C_2 P^t AP.
}
Similarly,
\be{equation}
{
\phi_i(A)=C_i P^t AP,\qquad\forall A\in D_n,\quad i=1,\ldots,m,
}
in which $C_1\cdots C_m=I$ by \eqref{eq D_n trace preserve m product}.
}

Both  Theorem \ref{thm: trace preserve 2 product}  and Theorem \ref{thm: D_n trace preserve m product} have counterparts in maps on
the set of real diagonal matrices and the proofs are analogeous.

%
%

\end{document}